\documentclass{amsart}
\usepackage{cases}
\usepackage{amsfonts}
\usepackage{mathrsfs}
\usepackage{amsmath}
\usepackage{amsfonts,amssymb,amsmath,amsthm}
\usepackage{url}
\usepackage{enumerate}
\usepackage{amscd}
\usepackage{mathrsfs,graphicx}
\usepackage{latexsym}
\usepackage[all]{xy}
\usepackage{epsfig}

\newtheorem*{thm*}{Theorem}
\newtheorem{thm}{Theorem}[section]
\newtheorem{lem}[thm]{Lemma}

\theoremstyle{definition}
\newtheorem{df}[thm]{Definition}

\newtheorem*{ques*}{Question}
\numberwithin{equation}{section}
\newcommand{\Ext}{\mbox{\rm Ext}}
\newcommand{\Hom}{\mbox{\rm Hom}}

\newcommand{\Ker}{\mbox{\rm Ker}}

\author{Gang Yang}\address{School of Mathematics, Physics and Software Engineering, Lanzhou Jiaotong University,  Lanzhou {\rm 730070}, P.R. China} \email{yanggang10@gmail.com}
\author{Li Liang}\address{School of Mathematics, Physics and Software Engineering, Lanzhou Jiaotong University,  Lanzhou {\rm 730070}, P.R. China} \email{lliangnju@gmail.com}
\thanks{This work was partly supported
by NSF of China (Grant No. 11101197) and the Program of Science and
Technique of Gansu Province (No. 1107RJZA233).}

\begin{document}

\title[All modules have Gorenstein flat precovers]{\Large All modules have Gorenstein flat precovers}

\begin{abstract}
It is known that every $R$-module has a flat precover. We show
in the paper that every $R$-module has a Gorenstein flat precover.\\
\end{abstract}
\maketitle

\section {\large\bf Introduction}
A class $\mathcal{L}$ of objects of an abelian category
$\mathcal{C}$ is called a precovering class \cite{Enoc81} if every
object of $\mathcal{C}$ has an $\mathcal{L}$-precover (see
Definition 2.2). In the language of \cite{AR} this means that
$\mathcal{L}$ is a contravariantly finite subcategory. Precovering
classes (or contravariantly finite subcategories) play a great
important role in homological algebra. One of the reasons is that
one can construct proper $\mathcal{L}$-resolutions using a
precovering class $\mathcal{L}$ to compute homology and cohomology
(see \cite{EJ00} for details).

For any ring $R$, recall from \cite{EJT93} that an $R$-module $G$ is
Gorenstein flat if there exists an exact sequence $\cdots\rightarrow
F^{-2}\rightarrow F^{-1}\rightarrow F^0\rightarrow F^1\rightarrow
F^2\rightarrow\cdots$ of flat $R$-modules with
$G=\text{Ker}(F^0\rightarrow F^1)$ such that $I\otimes_R-$ leaves
the sequence exact whenever $I$ is an injective right $R$-module.
Obviously, flat $R$-modules are Gorenstein flat. Further studies on
Gorenstein flat $R$-modules can be found in \cite{Ben09, EJ00,
EJLR04, EJT93, Holm04a}. Bican, El Bashir and Enochs \cite{BBE01}
proved that the class of flat $R$-modules is a precovering class. On
the other hand, Enochs, Jenda and L\'{o}pez-Ramos \cite{EJLR04}
proved that the class of Gorenstein flat $R$-modules is a
precovering class over a right coherent ring. Furthermore, it was
shown in \cite{YL} that the result holds over a left GF-closed ring
(that is, a ring over which the class of Gorenstein flat $R$-modules
is closed under extensions). In this paper, we prove that the class
of Gorenstein flat $R$-modules is a precovering class over any ring
as follows.

\vspace{0.3cm}\hspace{-0.4cm}\textbf{Theorem A.} \emph{Let $R$ be
any ring. Then every $R$-module has a Gorenstein flat precover.}
\vspace{0.3cm}

We prove the above result by constructing a perfect cotorsion pair
in the category of complexes of $R$-modules.

\section {\large\bf Preliminaries}
Throughout the paper, we assume all rings have an identity and all
modules are unitary. Unless stated otherwise, an $R$-module will be
understood to be a left $R$-module.

To every complex $\xymatrix@C=0.6cm{C= \cdots \ar[r]^{} & C^{m-1}
\ar[r]^{d^{m-1}} &
 C^m \ar[r]^{d^{m}} & C^{m+1} \ar[r]^{d^{m+1}} & \cdots
 },$ the $m$th cycle  is defined as $\Ker(d^m)$ and is denoted
$\text{Z}^m(C)$. The $m$th boundary  is $\text{Im}(d^{m-1})$ and is
denoted $\text{B}^m(C)$. The  $m$th  homology of $C$ is the module
$$\text{H}^m(C)=\text{Z}^m(C)/\text{B}^m(C).$$ A complex $C$ is exact
if $\text{H}^m(C)=0$ for all $m\in\mathbb{Z}$. For an integer $n$,
$C[n]$ denotes the complex such that $C[n]^m=C^{m+n}$ and whose
boundary operators are $(-1)^nd^{m+n}$.  Given an $R$-module $M$, we
denote by $\overline{M}$ the complex $$\xymatrix@C=0.6cm{
  \cdots \ar[r]^{ }  & 0 \ar[r]^{ } & M \ar[r]^{id} &
  M\ar[r]^{ } & 0 \ar[r]^{ } &  \cdots  }$$
with $M$ in the $-1$ and 0th degrees and $\underline{M}$ the complex
$$\xymatrix@C=0.6cm{
  \cdots \ar[r]^{ }  & 0 \ar[r]^{ }  &
  M \ar[r]^{ } & 0 \ar[r]^{ } &  \cdots  }$$ with $M$ in the $0$th
degree. A complex $C$ is finitely presented (generated) if only
finitely many components are nonzero and each $C^m$ is finitely
presented (generated). Clearly, both $\overline{R}$ and
$\underline{R}$ are finitely presented. Recall that a complex $P$ is
projective if it is exact and $\text{Z}^m(P)$ is a projective
$R$-module for each $m\in \mathbb{Z}$, so it is easy to see that $P$
is a direct sum  of the form $\overline{Q}[m]$ for some projective
$R$-modules $Q$. Given two complexes $X$ and $Y$, we let
$\Hom^\bullet(X, Y)$ denote a complex of $\mathbb{Z}$-modules with
$m$th component
$$\Hom^\bullet(X, Y)^m=\prod_{t\in \mathbb{Z}}\Hom(X^t, Y^{m+t})$$
and such that if $f\in\Hom^\bullet(X, Y)^m$ then
$$(d^m(f))^n=d_Y^{n+m}\circ f^n-(-1)^{m}f^{n+1}\circ d_X^n.$$
We say $f:X\rightarrow Y$ a morphism of complexes if $d_Y^{n}\circ
f^n=f^{n+1}\circ d_X^n$ for all $n\in \mathbb{Z}$. $\Hom(X, Y)$
denotes the set of morphisms of complexes from $X$ to $Y$ and
$\Ext^i(X, Y)$ $(i\geq1)$ are right derived functors of $\Hom$.
Obviously, $\Hom(X, Y)=\text{Z}^0(\Hom^\bullet(X, Y))$. We let
$\underline{\Hom}(X, Y)$ denote a complex with $\underline{\Hom}(X,
Y)^m$ the abelian group of morphisms from $X$ to $Y[m]$ and with a
boundary operator given by: $f\in\underline{\Hom}(X, Y)^m$, then
$d^m(f): X\rightarrow Y[m+1]$ with $d^m(f)^n=(-1)^md_Y\circ f^n$,
$\forall n\in \mathbb{Z}$. We note that the new functor
$\underline{\Hom}(X, Y)$ has right derived functors whose values
will be complexes. These values should certainly be denoted
$\underline{\Ext}^i(X, Y)$. It is not hard to see that
$\underline{\Ext}^i(X, Y)$ is the complex
$$\cdots\rightarrow{\Ext}^i(X, Y[n-1])\rightarrow
{\Ext}^i(X, Y[n])\rightarrow{\Ext}^i(X, Y[n+1])\rightarrow\cdots$$
with boundary operator induced by the boundary operator of $Y$.

If $X$ is a complex of right $R$-modules and $Y$ is a complex of
left $R$-modules, let $X\otimes^\bullet Y$ be the usual tensor
product of complexes. I.e., $X\otimes^\bullet Y$ is the complex of
abelian groups with
$$(X\otimes^\bullet Y)^m=\bigoplus_{t\in \mathbb{Z}}X^t\otimes_R Y^{m-t}$$
and $$d(x\otimes y)=d_X^t(x)\otimes y+(-1)^{t}x\otimes
d_Y^{m-t}(y)$$ for $x\in X^t$ and $y\in Y^{m-t}$. Obviously,
$\underline{M}\otimes^\bullet Y=M\otimes_R Y= \cdots\rightarrow
M\otimes_R Y^{-1}\rightarrow M\otimes_R Y^0\rightarrow M\otimes_R
Y^1\rightarrow\cdots$ for a right $R$-module $M$. We define
$X\otimes Y$ to be $\frac{(X\otimes^\bullet
Y)}{\text{B}(X\otimes^\bullet Y)}$. Then with the maps
$$\frac{(X\otimes^\bullet
Y)^n}{\text{B}^n(X\otimes^\bullet Y)}\rightarrow
\frac{(X\otimes^\bullet Y)^{n+1}}{\text{B}^{n+1}(X\otimes^\bullet
Y)}, \quad x\otimes y\mapsto d_X(x)\otimes y,
$$
where $x\otimes y$ is used to denote the coset in
$\frac{(X\otimes^\bullet Y)^n}{\text{B}^n(X\otimes^\bullet Y)}$, we
get a complex of abelian groups.

One can found the next result in  \cite[Proposition 4.2.1]{Garc99}.

\begin{lem} \label{l2.1} Let $X$, $Y$, $Z$ be complexes.  Then we have the following natural isomorphisms:
\begin{enumerate}
  \item[$(1)$] $X\otimes(Y\otimes Z)\cong (X\otimes Y)\otimes Z$;
  \item[$(2)$] For a right $R$-module $M$, $\overline{M}[n]\otimes Y\cong M\otimes_R Y[n]$;
  \item[$(3)$]$X\otimes (\varinjlim Y_i)\cong \varinjlim (X\otimes Y_i)$
  for a directed family $(Y_i)_{i\in I}$ of complexes.
\end{enumerate}
\end{lem}

\begin{df}
Let $\mathcal{L}$ be a class of objects of an abelian category
$\mathcal{C}$ and $X$ an object. A homomorphism $f: L\rightarrow X$
is called an $\mathcal{L}$-precover if $L\in\mathcal{L}$ and the
abelian group homomorphism $\text{Hom}(L', f): \text{Hom}(L',
L)\rightarrow \text{Hom}(L', X)$ is surjective for each
$L'\in\mathcal{L}$.   An $\mathcal{L}$-precover $f: L\rightarrow X$
is called an $\mathcal{L}$-cover if every endomorphism $g:
L\rightarrow L$ such that $fg=f$ is an isomorphism. Dually we have
the definitions of an $\mathcal{L}$-preenvelope and  an
$\mathcal{L}$-envelope.
\end{df}

\begin{df} A pair $(\mathcal{A},
\mathcal{B})$ in an abelian category $\mathcal{C}$ is called a
cotorsion pair if the following conditions
hold:
\begin{enumerate}
                \item $\Ext^1_\mathcal{C}(A, B)=0$ for all $A\in\mathcal{A}$ and $B\in\mathcal{B}$;
                \item If $\Ext^1_\mathcal{C}(A, X)=0$ for all $A\in\mathcal{A}$ then $X\in\mathcal{B}$;
                \item If $\Ext^1_\mathcal{C}(X, B)=0$ for all $B\in\mathcal{B}$ then $X\in\mathcal{A}$.
\end{enumerate}
\end{df}

We think of a cotorsion pair $(\mathcal{A}, \mathcal{B})$ as being
$\lq\lq$orthogonal with respect to $\Ext^1_\mathcal{C}$". This is
often expressed with the notation $\mathcal{A}={^\perp\mathcal{B}}$
and $\mathcal{B}=\mathcal{A}^\perp$. The notion of a cotorsion pair
was first introduced by Salce in \cite{S79} and rediscovered by
Enochs and coauthors in 1990's. Its importance in homological
algebra has been shown by its use in the proof of the existence of
flat covers of modules over any ring \cite{BBE01}.

\begin{df} A cotorsion pair $(\mathcal{A}, \mathcal{B})$ is
said to be complete if for any object $X$ there are exact sequences
$0\rightarrow X\rightarrow B\rightarrow A\rightarrow 0$ and
$0\rightarrow B'\rightarrow A'\rightarrow X\rightarrow 0$ with $A,
A'\in \mathcal{A}$ and $B, B'\in \mathcal{B}$.
\end{df}

\begin{df}  A cotorsion pair $(\mathcal{A},
\mathcal{B})$ is said to be cogenerated by a set if there is a set
$\mathcal{S}\subset \mathcal{A}$ such that
$\mathcal{S}^\bot=\mathcal{B}$.
\end{df}

By a well-known theorem of Eklof and Trlifaj \cite{ET01}, a
cotorsion pair $(\mathcal{A}, \mathcal{B})$ is complete if it is
cogenerated by a set (see \cite{BBE01}).

\begin{df} A cotorsion pair $(\mathcal{A}, \mathcal{B})$ is
said to be perfect if every object has an $\mathcal{A}$-cover and a
$\mathcal{B}$-envelope.
\end{df}

%%%%%%%%%%%%%%%%%%%%%%%%%%%%%%%%%%%%%%%%%%%%%%%%%%%%%%%%%%%%%%%%%%%%%%%%%%%%%%%%%%%%%%%%%%%%%%%%%%%%%%%%%%%%%%%%%%%%%%%%%%%%

\section{\bf All modules have Gorenstein flat precovers}\label{ns}
 Recall from \cite{Garc99} that an exact
sequence $0\rightarrow P\rightarrow X\rightarrow X/P\rightarrow 0$
of complexes is \emph{pure} if for any complex $Y$ of right
$R$-modules, the sequence $0\rightarrow Y\otimes P\rightarrow
Y\otimes X\rightarrow Y\otimes X/P\rightarrow 0$ is exact. We state
here the characterizations of purity that can be found in
\cite[Theorem 5.1.3]{Garc99}.

\begin{lem}\label{p2.1} Let
$0\rightarrow P\rightarrow X\rightarrow X/P\rightarrow 0$ be an
exact sequence of complexes.  Then the following statements are
equivalent.
\begin{enumerate}
                \item[$(1)$] $0\rightarrow P\rightarrow X\rightarrow X/P\rightarrow
               0$ is pure;
               \item[$(2)$]  $0 \rightarrow\underline{\Hom}(U, P)\rightarrow\underline{\Hom}(U, X)
                \rightarrow\underline{\Hom}(U, X/P)\rightarrow0$ is exact for any finitely presented complex
                $U$.
             \end{enumerate}
\end{lem}

Recall from \cite{AF91} that a complex $Q$ is DG-projective, if each
$R$-module $Q^m$ is projective and $\Hom^\bullet(Q, E)$ is exact for
any exact complex $E$. By \cite[Proposition 2.3.5]{Garc99}, a
complex $Q$ is DG-projective if and only if $\Ext^1(Q, E)=0$ for
every exact complex $E$.

\begin{lem}\label{l2.2} Let
$  0 \rightarrow P \rightarrow X\rightarrow X/P\rightarrow 0$ be a
pure exact sequence of complexes.  If $X$ is exact then both $P$ and
$X/P$ are also exact.
\end{lem}
\begin{proof}
By Lemma \ref{p2.1},  the sequence
$\underline{\Hom}(D,X)\rightarrow\underline{\Hom}(D,X/P)\rightarrow0$
is exact for all finitely presented complex $D$, and so the sequence
$$\underline{\Hom}
(\underline{R},X)\rightarrow\underline{\Hom}(\underline{R},X/P)\rightarrow0$$
is exact since $\underline{R}$ is finitely presented. On the other
hand, the sequence
$$\underline{\Hom}(\underline{R},X)\rightarrow\underline{\Hom}(\underline{R},X/P)\rightarrow
\underline{\Ext}^1(\underline{R},P)\rightarrow
\underline{\Ext}^1(\underline{R},X)$$ is exact, where
$\underline{\Ext}^1(\underline{R},X)=0$ since $\underline{R}$ is
DG-projective and $X$ is exact. Thus we get that
$\underline{\Ext}^1(\underline{R},P)=0$, and so
$\text{H}^{-n+1}(P)\cong \Ext^1(\underline{R},P[-n])=0$ for all
$n\in \mathbb{Z}$. This means that $P$ is an exact complex, and now
it is easily seen that $X/P$ is also exact.
\end{proof}

Let $R$ be a ring,  we denote by $\mathbf{E}(R)$ the class of exact
complexes  of flat $R$-modules such that they remain exact after
applying $I\otimes_R-$ for any injective right $R$-module $I$.
Recall that a complex $F$ is flat if  $F$ is exact and each
$\text{Z}^n(F)$ is a flat $R$-module for each $n\in \mathbb{Z}$.
Clearly, $\mathbf{E}(R)$ contains all flat complexes. As
characterized in \cite{Gill04} and \cite{Garc99}, there are initiate
connections between the purity and the flatness of complexes.
Inspired by this fact we give the following result.

\begin{lem}\label{l3.3}  Let $R$ be any
ring and  $E\in \mathbf{E}(R)$.  If $S\subseteq E$ is pure, then $S$
and $E/S$ are both in $\mathbf{E}(R)$.
\end{lem}

\begin{proof}  Let $M$ be any right $R$-module. Then
$$0\rightarrow\overline{M}[n]\otimes
S\rightarrow \overline{M}[n]\otimes E\rightarrow
\overline{M}[n]\otimes E/S\rightarrow0$$ is exact. By Lemma
\ref{l2.1}(2), the sequence
$$0\rightarrow M\otimes_R
S[n]\rightarrow M\otimes_R E[n]\rightarrow M\otimes_R
(E/S)[n]\rightarrow0$$ is exact. Therefore $S^n\subseteq E^n$ is
pure for each $n\in\mathbb{Z}$. Since each $E^n$ is flat, we get
that $S^n$ and $E^n/S^n$ are flat for each $n\in\mathbb{Z}$.

By Lemma \ref{l2.2}, we get that $S$ and $E/S$ are exact. It remains
to show that for any injective right $R$-module $I$, $I\otimes_RS$
and $I\otimes_RE/S$ are exact.

Since the exact sequence $0\rightarrow S\rightarrow E\rightarrow
E/S\rightarrow0$ is pure, we get that the sequence
$$0\rightarrow\overline{I}\otimes
S\rightarrow \overline{I}\otimes E\rightarrow \overline{I}\otimes
E/S\rightarrow0$$ is exact and pure by Lemma \ref{l2.1}(1). Note
that $\overline{I}\otimes E\cong I\otimes_R E$ is exact by Lemma
\ref{l2.1}(2), then $\overline{I}\otimes S$ and $\overline{I}\otimes
E/S$ are exact by Lemma \ref{l2.2}, and so $I\otimes_RS$ and
$I\otimes_RE/S$ are exact by Lemma \ref{l2.1}(2).
\end{proof}

\begin{lem}\label{l2.4}  Let ${\rm Card}(R)\leq \kappa$,
where $\kappa$ is some infinite cardinal. Then for any
$F\in\mathbf{E}(R)$ and any element $x\in F$ (by this we mean $x\in
F^n$ for some $n$), there exists a subcomplex $L\subseteq F$ with
$x\in L$, $L, F/L\in \mathbf{E}(R)$ and  ${\rm Card}(L)\leq \kappa$.
\end{lem}

\begin{proof} By \cite[Lemma 4.6]{Gill04},  there exists a pure subcomplex $L\subseteq F$ with
$x\in L$ and  ${\rm Card}(L)\leq \kappa$, then, by Lemma \ref{l3.3},
we get that  $L$ and  $F/L$ are contained in  $\mathbf{E}(R)$.
\end{proof}

\begin{lem}\label{l3.5} For any ring $R$ the pair $(\mathbf{E}(R), \mathbf{E}(R)^\bot)$ is a perfect cotorsion
pair.
\end{lem}

\begin{proof} By Lemma \ref{l2.1}(3) the class $\mathbf{E}(R)$ is closed under direct
limits. Clearly, $\mathbf{E}(R)$ is closed under direct sums, direct
summands and extensions. Using Lemma \ref{l2.4} and a similar method
as proved in \cite[Remark 3.2]{AE01}, we get that the pair
$(\mathbf{E}(R), \mathbf{E}(R)^\bot)$ is cogenerated by a set. On
the other hand, the class $\mathbf{E}(R)$ contains all projective
complexes. Thus, by \cite[Corollaris 2.11, 2.12 and 2.13]{AE01}, the
pair $(\mathbf{E}(R), \mathbf{E}(R)^\bot)$ is a perfect cotorsion
pair.
\end{proof}

\vspace{0.3cm}

\hspace{-0.5cm}\emph{Proof of Theorem A.} Let $M$ be any $R$-module
and $g: E\rightarrow \underline{M}[1]$ be an
$\mathbf{E}(R)$-precover  which exists by Lemma \ref{l3.5}. This
gives the following commutative diagram:
$$
\xymatrix@C=15pt@R=30pt{
E=: \ \cdots \ar[r]^{}   &E^{-2} \ar[rr]^{}\ar[dd]^{}  & & E^{-1} \ar[dr]_{\pi}\ar[rr]^{} \ar[dd]^{g^{-1}}   &  & \ E^0 \ar[rr]^{}\ar[dd]^{} &&  E^1 \ar[rr]^{}\ar[dd]^{} & & \cdots       \\
& \ \ & & \ &  G \ar[dd]^{\widetilde{g}} \ar@{.>}[ur]^{}  \\
\underline{M}[1]=: \ \cdots \ar[r]^{}&0 \ar[rr]^{} & &M
\ar[dr]_{=}\ar[rr]_{}
& & 0 \ar[rr]^{} && 0 \ar[rr]^{} & & \cdots \\
 & \ \ & &\       &  M \ar[ur]^{}          }
$$
where $G=\text{Z}^0(E)$ is Gorenstein flat. In the following we show
that $\widetilde{g}: G\rightarrow M$ is a Gorenstein flat precover
of $M$.

Let $\widetilde{f}: H\rightarrow M$ be a homomorphism with $H$
Gorenstein flat. Then there exists a complex  $F$ in $\mathbf{E}(R)$
such that $H=\text{Z}^0(F)$. Now one can extend $\widetilde{f}$ to a
morphism $f: F\rightarrow \underline{M}[1]$ of complexes as follows:
$$
\xymatrix@C=15pt@R=30pt{
F=: \ \cdots \ar[r]^{}   &F^{-2} \ar[rr]^{}\ar[dd]^{}  & & F^{-1} \ar[dr]_{\sigma}\ar[rr]^{} \ar[dd]^{f^{-1}}   &  & \ F^0 \ar[rr]^{}\ar[dd]^{} &&  F^1 \ar[rr]^{}\ar[dd]^{} & & \cdots       \\
& \ \ & & \ &  H \ar[dd]^{\widetilde{f}} \ar@{.>}[ur]^{}  \\
\underline{M}[1]=: \ \cdots \ar[r]^{}&0 \ar[rr]^{} & &M
\ar[dr]_{=}\ar[rr]_{}
& & 0 \ar[rr]^{} && 0 \ar[rr]^{} & & \cdots \\
 & \ \ & &\       &  M \ar[ur]^{}          }
$$

Since $g: E\rightarrow \underline{M}[1]$ is an
$\mathbf{E}(R)$-precover, there exists a morphism $h: F\rightarrow
E$ of complexes such that the diagram $$\xymatrix{
                &        E \ar[d]^{g}     \\
  F \ar[ur]^{h} \ar[r]_{f} & \underline{M}[1]            }$$ is
  commutative.

The morphism  $h$ induces a homomorphism $\widetilde{h}:
H\rightarrow G$ such that the following diagram
$$
\xymatrix@C=15pt@R=30pt{
F=: \ \cdots \ar[r]^{}   &F^{-2} \ar[rr]^{}\ar[dd]^{h^{-2}}  & & F^{-1} \ar[dr]_{\sigma}\ar[rr]^{} \ar[dd]^{h^{-1}}   &  & \ F^0 \ar[rr]^{}\ar[dd]^{h^0} &&  F^1 \ar[rr]^{}\ar[dd]^{h^1} & & \cdots       \\
& \ \ & & \ &  H \ar[dd]^{\widetilde{h}} \ar@{.>}[ur]^{}  \\
E=: \ \cdots \ar[r]^{}&E^{-2} \ar[rr]^{} & &E^{-1}
\ar[dr]_{\pi}\ar[rr]_{}
& & E^0 \ar[rr]^{} && E^1 \ar[rr]^{}  && \cdots \\
 & \ \ & &\       &  G \ar[ur]^{}          }
$$
is commutative. Note that
$\widetilde{f}\sigma=f^{-1}=g^{-1}h^{-1}=\widetilde{g}\pi
h^{-1}=\widetilde{g}\widetilde{h}\sigma,$ then
$\widetilde{f}=\widetilde{g}\widetilde{h}$ since $\sigma$ is an
epimorphism.  This implies that  $\widetilde{g}: G\rightarrow M$ is
a Gorenstein flat precover of $M$.

\end{document}